\documentclass[reqno,11pt]{amsart}
\usepackage{latexsym,amsfonts,amssymb,amsmath,euscript}
\usepackage[latin1]{inputenc}
\usepackage{graphicx}
\usepackage{hyperref}
\usepackage{color,fancybox}
\usepackage[all]{xy}

\newcommand{\R}{\mathbb{R}}

\numberwithin{equation}{section}

\newtheorem{lemma}{Lemma}[section]
\newtheorem{proposition}{Proposition}[section]

\newtheorem{corollary}{Corollary}[section]

\title[Upper semicontinuity of global attractors]{Upper semicontinuity of global attractors for parabolic equations governed by the p-laplacian on unbounded thin domains}

\author[R. P. Silva]{Ricardo P. Silva}
\address[R. P. Silva]{Univ Estadual Paulista, Instituto de Geoci\^{e}ncias e Ci\^{e}ncias Exatas, Departamento de Matem\'atica, 13506-900, Rio Claro, SP, Brazil}
\email{rpsilva@rc.unesp.br}

\begin{document}

\maketitle

\begin{abstract}
We consider the asymptotic behavior of quasilinear parabolic equations posed in a family of unbounded domains that degenerates onto a lower dimensional set. Considering an au\-xi\-liary family of weighted Sobolev spaces we show the existence of global attractors and we analyze convergence properties of the solutions as well of the attractors.
\end{abstract}

\

{\footnotesize
\subjclass{{ Mathematics Subject Classifications: 35B25, 35B40, 35B41}

\keywords{Keywords: thin domains; p-laplacian; global attractors}}  }

\

\section{Introduction}\label{sec:introd}

The systematic study of the asymptotic behavior of dissipative systems on thin domains started with the works \cite{Hale:92, HR:92a} by J. Hale and G. Raugel. The cern of the study is guided by the question: Is it possible to give some information on the dynamics of an evolution equation defined in a spatial domain which is \emph{small} in some direction by mean a model on a lower dimensional spatial domain?  If such systems possess global attractors then is possible to compare the asymptotic behavior of two semiflows in terms of the Hausdorff distance of their respective attractors.

There is an extensive bibliography on thin domain problems especially devoted on the reaction-diffusion model
\begin{equation}\label{eq:prot-semi-lin}
\begin{array}{ll}
u_t - \Delta u + \lambda u = f(u), & \text{ in } \quad  (0,\infty) \times \Omega^\epsilon, \\

\hspace{0.2cm} \dfrac{\partial u}{\partial \eta_\epsilon} = 0, & \text{ on } \quad  (0,\infty) \times \partial \Omega^\epsilon,
\end{array}
\end{equation}
where $\Omega^\epsilon$ is a family of bounded domains collapsing onto a lower dimensional subset.

In \cite{Hale:92}, Hale and Raugel considered the case of domain of the form
$$
\Omega^\epsilon := \{(x, \epsilon y) \in \R^{n} \times \R : x \in \omega, \, 0<y< g(x) \},
$$
where $\omega$ is a bounded domain and $g$ is a smooth positive function defined on $\omega$. When $\epsilon$ is small, they compare the dynamics of \eqref{eq:prot-semi-lin} with the dynamics of the following equation defined in $\omega$
\begin{equation}\label{eq:hale-limit}
\begin{array}{ll}
u_t - \dfrac{1}{g}{\rm div}(g \nabla u)+ \lambda u = f(u), & \text{ in } \quad  (0,\infty) \times \omega, \\

\hspace{0.2cm} \dfrac{\partial u}{\partial \eta} = 0, & \text{ on } \quad  (0,\infty) \times \partial \omega.
\end{array}
\end{equation}
In particular they proof that the family of global attractors $\mathcal{A}_\epsilon$ associated to \eqref{eq:prot-semi-lin} is upper semicontinuous in $\epsilon=0$.

M. Prizzi and K. Rybakowski in \cite{Prizzi:01} treated a much more general class of thin domains, namelly
\begin{equation}\label{eq:domain-pri-rib}
\Omega^\epsilon := \{(x,\epsilon y) \in \R^{m} \times \R^n : (x,y) \in \omega \},
\end{equation}
where $\omega \subset \R^{m+n}$ is a bounded domain. They developed an abstract framework for the analysis of the problem \eqref{eq:prot-semi-lin} and they also shows the upper semicontinuity of the global attractors.

In \cite{Antoci:01}, F. Antonci and M. Prizzi allowed in \eqref{eq:domain-pri-rib} the domain $\omega$ to be an unbounded set. In such case the compactness of the semiflows is lost but inspired by \cite{Wang:99} the authors were able to show the existence of global attractors as well its upper semicontinuity.

Associated with boundary oscillation (rough boundary) on thin structures, J. Arrieta et al. in \cite{ACPS} consider 
$$
\Omega^\epsilon := \{(x, \epsilon y) \in \R \times \R : x \in (0,1), \ 0< y <  g(\epsilon^{-1} x))  \},
$$
where $g: \R \to \R$ is a $L$-periodic function. Combining methods from homogenization theory \cite{CP:80} the authors showed that the limiting equation is
$$
\begin{array}{ll}
u_t - r u_{xx} + \lambda u = f(u), & \text{ in } \quad  (0,\infty) \times (0,1), \\

u_x(t,0) = u_x(t,1) = 0, &  t>0,
\end{array}
$$
where $r>0$ is called the \emph{homogenized coefficient}. In particular the authors also show the upper semicontinuity of global attractors. For more references we refer the reader to the Montecatini lecture notes  \cite{Raugel:95} by G. Raugel.

Despite the study of the asymptotic behavior for semilinear models be widely considered in the literature, the same is not true for the quasi-linear case. Therefore, in this paper we consider an evolution equation governed by the $p$-laplacian operator as prototype of quasi-linear equations on an unbounded thin domain of the form 
$$
\Omega^\epsilon := \{(x, \epsilon y) \in \R^{n} \times \R : 0<y< g(x) \}.
$$

Considering in $\Omega^\epsilon$ the family of quasi-linear evolution equations
$$
\begin{array}{ll}
u_t -\Delta_p u + a(x,\epsilon y) |u|^{p-2} u = f(u), & \quad \hbox{in } (0,\infty) \times \Omega^\epsilon, \\
\dfrac{\partial u}{\partial \eta^\epsilon} = 0, &  \quad \hbox{on } (0,\infty) \times \partial \Omega^\epsilon,
\end{array}
$$
where $\Delta_p u:= {\rm div}(|\nabla u|^{p-2} \nabla u)$ denotes the $p$-Laplacian operator, $2<p<n$, we will compare the semiflow generated by them with the semiflow generated by the following equation (see \cite{S:13a})
$$
u_t -\dfrac{1}{g} {\rm div}(g |\nabla u|^{p-2} \nabla u)  + a(x,0) |u|^{p-2} u = f(u), \quad \hbox{in } (0,\infty) \times \R^n.
$$
Notice that in the case $p=2$ the structure of the main part of the limiting problem agrees with Hale's and Raugel's limiting problem \eqref{eq:hale-limit}. Our aim is to prove existence of global attractors $\mathcal{A}_\epsilon$ as well its upper semicontinuity in $\epsilon=0$.

For the best of our knowledge this is the first attempt to consider the asymptotic behavior of quasi-linear parabolic equations on unbounded thin domains. Such model is relevant in a variety of physical phenomena as non-Newtonian fluids and in flow through porous media.

\

The paper is organized as follows: In Section \ref{sec:setting} we describe the perturbed problems as well the nature of the limiting one. In Section \ref{sec:prel} we present the functional setting and the abstract formulation of the problem. In the Section \ref{sec:uniform} we prove uniform dissipation of the semigroups which leads to the existence of global attractors as well their upper-semicontinuity.


\section{Setting of the problem}\label{sec:setting}

Given a bounded function $g \in C^2 (\R^n; \R)$ satisfying $0 < \alpha_1 \le g(x) \le \alpha_2$, $\forall \, x \in \R^n$, and a positive parameter $\epsilon \in (0,1)$, we consider the family of \emph{unbounded thin domains}, $\Omega^\epsilon \subset \R^{n+1}$, defined by
\begin{equation}\label{eq:domain-pert-unbound}
\Omega^\epsilon := \{(x,y) \in \R^{n} \times \R : \, 0<y< \epsilon g(x) \}.
\end{equation}

In $\Omega^\epsilon$ we consider the family of  quasilinear parabolic equations
\begin{equation}\label{eq:prot-pert}
\begin{array}{ll}
u^\epsilon_t -\Delta_p u^\epsilon + a |u^\epsilon |^{p-2} u^\epsilon = f(u^\epsilon), & \quad \hbox{in } (0,\infty) \times \Omega^\epsilon, \\
 \dfrac{\partial u^\epsilon}{\partial \eta^\epsilon} = 0, &  \quad \hbox{on } (0,\infty) \times \partial \Omega^\epsilon, \\
 u^\epsilon(0, \cdot) = u^\epsilon_0(\cdot) \in L^2(\Omega^\epsilon),
\end{array}
\end{equation}
where $\Delta_p u:= {\rm div}(|\nabla u|^{p-2} \nabla u)$ denotes the $p$-Laplacian operator, $2<p<n$,    $\eta^\epsilon$ denotes the outward unitary normal vector field to $\partial \Omega^\epsilon$ and $f:\R \to \R$ is a globally Lipschitz function such that $f(0)=0$. We assume that $a:\R^{n} \times \R \to \R$ is an uniformly continuous function satisfying $a(x,y) \geq 1$ for all $(x,y) \in \R^{n}\times \R$, and
\begin{equation}\label{eq:a0}
\int_{\R^{n+1}} \frac{1}{a(x,y)^{\frac{2}{p-2}}}  dxdy < \infty.
\end{equation}

This technical hypothesis is necessary for the compactness embeddings of some weigthed Sobolev spaces which is essential in our approach to existence of global attractors for the problem \eqref{eq:prot-pert}.

In order to formulate the appropriate limiting regime for \eqref{eq:prot-pert}, we consider $\bar{a}:\R^n \to \R$ defined by $\bar{a}(x)= a(x,0)$ which also satisfies $\bar{a}(x) \geq 1$ for all $x \in \R^{n}$ and
$$
\int_{\R^{n}} \frac{1}{\bar{a}(x)^{\frac{2}{p-2}}}  dx < \infty.
$$

As indicated in \cite{S:13a}, the limiting equation should be
\begin{equation}\label{eq:limit-parab}
\begin{array}{ll}
u_t -\dfrac{1}{g} {\rm div}(g |\nabla u|^{p-2} \nabla u)  + \bar{a} |u|^{p-2} u = f(u), & \quad \hbox{in } (0,\infty) \times \R^n,   \\
u(0,\cdot) = u_0(\cdot) \in L^2(\R^n).
\end{array}
\end{equation}
%

Since we need to compare functions defined in $\Omega^\epsilon \subset \R^{n+1}$ with functions defined in $\R^n$ we will need some operators which will transform functions defined in $\R^n$ to functions defined in $\Omega^\epsilon$, as well operators which will transform functions defined in $\Omega^\epsilon$ to functions defined in $\R^n$.  Due to the nature of this specific kind of perturbation, is natural to consider the following operators
\begin{enumerate}
\item[] (\emph{Average projector})
\begin{equation}\label{eq:oper-restr}
\begin{array}{c}
M_\epsilon: L^p(\Omega^\epsilon) \to L^p(\R^n) \\
\displaystyle (M_\epsilon u)(x) = \frac{1}{\epsilon g(x)}\int_0^{\epsilon g(x)}  u(x,y) \, dy
\end{array}
\end{equation}
\item[] (\emph{Extension operator})
\begin{equation}\label{eq:oper-ext}
\begin{array}{c}
E_\epsilon: L^p(\R^n) \to L^p(\Omega^\epsilon) \\
(E_\epsilon u)(x,y) = u(x)
\end{array}
\end{equation}
\end{enumerate}
Notice that $M_\epsilon \circ E_\epsilon = I$, the identity operator in $L^p(\R^n)$. 

Furthermore the extension operator $E_\epsilon$ maps the family of spaces $W^{1,p}(\R^n)$ into $W^{1,p}(\Omega^\epsilon)$ and satisfies $\dfrac{\partial}{\partial y} (E_\epsilon u) = 0$.


\section{Functional Framework}\label{sec:prel}

In this section we recall definitions of suitable spaces and operators and some of their pro\-per\-ties. We start recalling that $\Omega^\epsilon \subset \R^{n+1}$ varies in accordance with the parameter $\epsilon$ collapsing themselves to lower dimension domain $\R^n$. Therefore, in order to preserve the ``relative capacity'' of a mensurable subset $\mathcal{O} \subset \Omega^{\epsilon}$, we rescale the Lebesgue measure by ${1}/{\epsilon}$, dealing with the singular measure $\rho_{\epsilon}(\mathcal{O})=\epsilon^{-1} | \mathcal{O} |$. With this measure we introduce the Lebesgue $L^p(\Omega^\epsilon; \rho_{\epsilon})$ and the Sobolev $W^{1,p}(\Omega^\epsilon;\rho_{\epsilon})$ spaces. The norms in these spaces will be denoted by $||| \cdot |||_{L^p(\Omega^\epsilon)}$ and $||| \cdot |||_{W^{1,p}(\Omega^\epsilon)}$ respectively and they are related with the usual ones by
$$
\begin{gathered}
||| u |||_{L^p(\Omega^\epsilon)} = \epsilon^{-1/p} \| u \|_{L^p(\Omega^\epsilon)}, \quad \forall \, u \in L^p(\Omega^\epsilon) \\
||| u |||_{W^{1,p}(\Omega^\epsilon)} = \epsilon^{-1/p} \| u \|_{W^{1,p}(\Omega^\epsilon)}, \quad \forall \,u \in W^{1,p}(\Omega^\epsilon).
\end{gathered}
$$
In the particular case $p=2$, we will denote the inner product as $\displaystyle \langle u,v \rangle_\epsilon:= \epsilon^{-1} \int_{\Omega^\epsilon} uv \,dxdy$.

In the limiting case, we consider the Lebesgue and Sobolev spaces $L^p(\R^n)$ and $W^{1,p}(\R^n)$ endowed with equivalent norms $\displaystyle |||u |||_{L^p(\R^n)} := \left[ \int_{\R^n} g|u|^p dx \right]^{\frac1p}$ and  $\displaystyle |||u |||_{W^{1,p}(\R^n)} := \left[ \int_{\R^n} g( |\nabla u|^p + |u|^p) dx\right]^{\frac1p}$ res\-pec\-tively.

With this definitions, it is easy to see from Fubini-Tonelli Theorem and Hölder inequality that the operators $M_\epsilon: L^p(\Omega^\epsilon) \to L^p(\R^n)$ satisfy $\| M_\epsilon \|_{\mathcal{L}(L^p(\Omega^\epsilon;\rho_{\epsilon}), L^p(\R^n))} = 1$. In fact, let $u \in L^p(\Omega^\epsilon)$,
\begin{align*}
|||M_\epsilon u |||_{L^p(\R^n)} & =  \left[ \int_{\R^n} g(x)|M_\epsilon u(x)|^p dx \right]^{\frac1p} \\
& =  \left[ \int_{\R^n} \frac{1}{\epsilon^p g^{p-1}(x)}\left|\int_0^{\epsilon g(x)} u(x,y) dy\right|^p dx \right]^{\frac1p} \\
& \le  \left[ \int_{\R^n} \frac{1}{\epsilon^p g^{p-1}(x)} (\epsilon g(x))^{p-1} \int_0^{\epsilon g(x)} |u(x,y)|^p dydx \right]^{\frac1p} \\
& = \epsilon^{-1/p}\left[  \int_{\Omega^\epsilon} |u(x,y)|^p dxdy  \right]^{\frac1p} = |||u |||_{L^p(\Omega^\epsilon)}. 
\end{align*}
The equality holds if $u$ is independent of $y$ in $\Omega^\epsilon$.

One of the major difficulties on the analysis of the asymptotic behavior of PDE's in unbounded domains, is the lack of compactness of the embeddings in some functional spaces, in our approach this will be fixed up introducing the family of weighted Sobolev spaces
$$ 
E_\epsilon = \{ u \in W^{1,p}(\Omega^\epsilon): \displaystyle \int_{\Omega^\epsilon} a  |u|^p \, dxdy  <~\infty \}
$$
endowed with norm $\| u \|_{E_\epsilon } = \left[\epsilon^{-1} \displaystyle \int_{\Omega^\epsilon} \left( |\nabla u|^p + a |u|^p  \right) \, dxdy \right]^\frac{1}{p} $. In the li\-mi\-ting case we consider
$$ 
E_0 = \{ u \in W^{1,p}(\R^n): \displaystyle \int_{\R^n} \bar{a}  |u|^p \, dx  < \infty \}
$$ 
endowed with norm $\| u \|_{E_0 } = \left[\displaystyle \int_{\R^n} g \left( |\nabla u|^p + \bar{a} |u|^p  \right) \, dxdy \right]^\frac{1}{p}$.

In the next Lemma we summarize some important properties of the family $E_\epsilon$.

\begin{lemma}\label{lemma:E}
The space $E_\epsilon$ is a reflexive Banach space. Furthermore $E_\epsilon \stackrel{d}{\hookrightarrow} L^r(\Omega^\epsilon)$, $2 \leq r \leq  \frac{p (n+1)}{(n+1)-p}$, and $E_\epsilon \subset \subset L^r(\Omega^\epsilon)$, $2 \leq r < \frac{p (n+1)}{(n+1)-p}$, with all embedding constants independent of $\epsilon$.

Similarly the space $E_0 $ is a reflexive Banach space, $E_0 \stackrel{d}{\hookrightarrow} L^r(\R^n)$, $2 \leq r \leq \frac{p n}{n - p}$, and $E_0 \subset \subset L^r(\R^n)$, $2 \leq r < \frac{p n}{n - p}$.
\end{lemma}
\begin{proof}
See \cite[Lemma 2.1]{CSS:13}.

\end{proof}

\subsection{Monotone operator}

In order to rewrite the problems \eqref{eq:prot-pert} and \eqref{eq:limit-parab} as evolution equations in the state space $L^2(\Omega^\epsilon)$ and $L^2(\R^n)$ respectively, we consider the following  operators 
$$
\begin{array}{c}
A_\epsilon : E_\epsilon \to E_\epsilon^* \\
\displaystyle \langle A_\epsilon u, v \rangle_{E^*_\epsilon,E_\epsilon} =  \int_{\Omega^\epsilon} \left(  |\nabla u|^{p-2} \nabla u \cdot \nabla v  + a  | u |^{p-2} u v \right) dxdy, \quad \forall \, v \in E_\epsilon,
\end{array}
$$
where $\langle \cdot, \cdot \rangle_{E^*_\epsilon,E_\epsilon} $ denotes the pair of duality between $E^*_\epsilon$ and $E_\epsilon$, and
$$
\begin{array}{c}
A_0 : E_0 \to E_0^* \\
\displaystyle \langle A_0 u, v \rangle_{E^*_0,E_0} = \int_{\R^n} \left(  |\nabla u|^{p-2} \nabla u \cdot \nabla v  + \bar{a}  | u |^{p-2} u v \right) dx, \quad \forall \, v \in E_0,
\end{array}
$$
where $\langle \cdot, \cdot \rangle_{E^*_0,E_0} $ denotes the pair of duality between $E^*_0$ and $E_0$.

By Tartar's inequality, \cite{Vrabie:87}, one can show that the operator $A_\epsilon$, $\epsilon \in [0,1)$  is monotone, hemicontinuous and coercive. Let us to consider the $L^r$-realization, $2\leq r \leq \frac{p (n+1)}{(n+1)-p}$, of the operator $A_\epsilon$, $\epsilon>0$, denoted by $A_{\epsilon,r}$, given by
\begin{eqnarray*}
D(A_{\epsilon,r}) & = &  \{ u \in E_\epsilon :A_{\epsilon} u \in L^r(\Omega^\epsilon)  \}, \\ 
A_{\epsilon,r} u & = & A_{\epsilon} u, \quad \forall \, u \in D(A_{\epsilon,r})
\end{eqnarray*}
as well the $L^r$-realization, $2\leq r \leq \frac{pn}{n-p}$, of the operator $A_0$, denoted by $A_{0,r}$, given by
\begin{eqnarray*}
D(A_{0,r}) & = &  \{ u \in E_0 :A_{0} u \in L^r(\R^n)  \}, \\ 
A_{0,r} u & = & A_{0} u, \quad \forall \, u \in D(A_{0,r}).
\end{eqnarray*}


For our purposes it is of special interest the case $r=2$, and for shorten notation, we will drop the index $r$ and we will write $A_\epsilon$ and $A_0$ for the respective realizations. 

We will also denote by $B_\epsilon:L^2(\Omega^\epsilon) \to L^2(\Omega^\epsilon)$ $(B: L^2(\R^n) \to L^2(\R^n)$ if $\epsilon =0)$ the Nemitskii operator associated to $f$. If $L$ is the Lipschitz constant of $f$,  it is easy to see that $B_\epsilon$ and  $B_0$ are Lipschitz maps with constant of Lipschitz $L$ and $B_\epsilon(0) =B(0)=0$.

With this framework the problems \eqref{eq:prot-pert} and \eqref{eq:limit-parab} can be written  respectively as
\begin{eqnarray}\label{eq:ev-abst}
u_t^\epsilon + A_\epsilon u^\epsilon  &=&  B_\epsilon(u^\epsilon), \quad t>0, \\
u^\epsilon(0)  &=&  u^\epsilon_0 \in L^2(\Omega^\epsilon), \nonumber
\end{eqnarray}
and
\begin{eqnarray}\label{eq:ev-abst-lim}
u_t + A_0 u  & = &  B(u), \quad t>0, \\
u(0) & = & u_0 \in L^2(\R^n). \nonumber
\end{eqnarray}

Solutions of these equations are obtained from the following Proposition

\begin{proposition}[\cite{Brezis:73}, Proposition 3.13]\label{teo:Brezis}
Given $u^\epsilon_0 \in L^2(\Omega^\epsilon)$ $(u_0 \in L^2(\R^n)$ if $\epsilon =0)$ there exists an unique solution $u^\epsilon =u^\epsilon(\cdot, u^\epsilon_0) \in W^{1,1}(0,\infty; L^2(\Omega^\epsilon))$ of \eqref{eq:ev-abst} $(u =u(\cdot, u_0) \in W^{1,1}(0,\infty; L^2(\R^n))$ of \eqref{eq:ev-abst-lim} if $\epsilon =0)$.
\end{proposition}

%


\section{Uniform Dissipativness and existence of Attractors}\label{sec:uniform}

In this section we establish uniform bounds on solutions of the problem \eqref{eq:ev-abst} and \eqref{eq:ev-abst-lim}. We will make use of a variation of the Gronwall's Inequality \cite{Teman:88}.

\begin{lemma}\label{lemma:l2bound} 
If $u^\epsilon(\cdot, u^\epsilon_0) \in C([0, \infty), L^2(\Omega^\epsilon))$ is the global solution of the problem \eqref{eq:ev-abst}, then there exist positive constants $T$ and $\beta$ $($not dependent on $\epsilon$$)$, such that
$$
||| u^\epsilon(t, u^\epsilon_0) |||_{L^2 (\Omega^\epsilon)} \leq \beta, \quad \forall \; t \geq T .
$$
\end{lemma}
\begin{proof}
Taking the $\langle \cdot, \cdot \rangle_\epsilon$ inner product with $u^\epsilon$ in  \eqref{eq:ev-abst} we obtain
\begin{equation}\label{eq:intbypart}
\frac{1}{2} \frac{d}{dt} ||| u^\epsilon |||^2_{L^2(\Omega^\epsilon)}  + \langle A_\epsilon u^\epsilon, u^\epsilon \rangle_\epsilon = \langle B_\epsilon (u^\epsilon) - B_\epsilon(0), u^\epsilon \rangle_\epsilon.
\end{equation}
Therefore
\begin{equation*}
\frac{1}{2} \frac{d}{dt} ||| u^\epsilon |||^2_{L^2(\Omega^\epsilon)}  +  c ||| u^\epsilon |||^p_{L^2(\Omega^\epsilon)} \leq L ||| u^\epsilon |||^2_{L^2(\Omega^\epsilon)}, 
\end{equation*}
where $c^{-1} > 0$ is the embedding constant in the Lemma \ref{lemma:E}.

Taking $\theta = \frac{p}{2}$, it follows from Young's inequality that for all $\eta > 0$,
\begin{eqnarray*}
\frac{1}{2} \frac{d}{dt} ||| u^\epsilon |||^2_{L^2(\Omega^\epsilon)}  +  c ||| u^\epsilon |||^p_{L^2(\Omega^\epsilon)} \leq \frac{1}{\theta'}\left( \frac{L}{\eta} \right)^{\theta'} +  \frac{\eta^\theta}{\theta} ||| u^\epsilon |||^p_{L^2(\Omega^\epsilon)}, 
\end{eqnarray*}
where $\theta' = \frac{p}{p-2}$. Choosing $\eta > 0$ satisfying $\gamma = c - \frac{\eta^\theta}{\theta}  > 0$, we have by \cite[Chapter III, Lemma 5.1]{Teman:88} that
$$
\frac{1}{2} ||| u^\epsilon |||^2_{L^2(\Omega^\epsilon)} \leq \left( \frac{ \delta}{\gamma} \right)^{\frac2p} + \left[ \frac{\gamma}{2}(p-2)t\right]^{\frac{-2}{p-2}},
$$
where $ \delta = \frac{1}{\theta'}\left( \frac{L}{\eta} \right)^{\theta'}$. Taking $T > 0$ in order to $\left[ \frac{\gamma}{2}(p-2)T\right]^{\frac{-2}{p-2}} \leq 1$,
we have for $ t \geq T$ that
$$
\frac{1}{2} ||| u^\epsilon |||^2_{L^2(\Omega^\epsilon)} \leq\left( \frac{ \delta}{\gamma} \right)^{\frac2p}  +1 : = \beta.
$$

\end{proof}

\begin{lemma}\label{lemma:Ebound}
If $u^\epsilon(\cdot, u^\epsilon_0) \in C([0, \infty), L^2(\Omega^\epsilon))$ is the global solution of \eqref{eq:ev-abst}, then there exist positive constants $T_1$ and $\beta_1$ $($not dependent on $\epsilon$$)$, such that
$$
\| u^\epsilon(t, u^\epsilon_0) \|_{E_\epsilon} \leq \beta_1, \quad \forall \; t \geq T_1 .
$$
\end{lemma}
\begin{proof}
Taking the $\langle \cdot, \cdot \rangle_\epsilon$ inner product with $u_t^\epsilon$ in \eqref{eq:ev-abst}  we obtain from Young's inequality that
$$ 
\frac12 ||| u_t^\epsilon |||^2_{L^2(\Omega^\epsilon)} + \frac{1}{p} \frac{d}{dt} \| u^\epsilon \|_{E_\epsilon}^p  \leq \frac{1}{2} L^2 ||| u^\epsilon |||^2_{L^2(\Omega^\epsilon)},
$$
and consequently, for $\theta= \frac{p}{2}$, we obtain
\begin{equation}\label{eq:limbound} 
\frac{2}{p} \frac{d}{dt} \| u^\epsilon \|_{E_\epsilon}^p \leq L^2 ||| u^\epsilon |||_{L^2(\Omega^\epsilon)}^2  \leq \frac{1}{\theta'} L^{2 \theta'} + \frac{ 1 }{\theta} \| u^\epsilon \|_{E_\epsilon}^p .
\end{equation}

Recalling \eqref{eq:intbypart} we have by integrating from $t$ to $t+R$ that
$$
\int_{t}^{t+R} \|u^\epsilon(s) \|^p_{E_\epsilon} \, ds \le \frac{1}{2} |||u^\epsilon(t)|||^2_{L^2(\Omega^\epsilon)} + L \int_t^{t+R} |||u^\epsilon(s)|||^2_{L^2(\Omega^\epsilon)} ds \le a_3.
$$

Setting $a_1= \frac{ R p }{2 \theta}$, $a_2 =  \frac{Rp}{2 \theta'} L^{2 \theta'}$ we obtain from  \eqref{eq:limbound} thanks to \cite[Chapter III, Lemma 1.1]{Teman:88}, that

%
%

\begin{equation}
\| u^\epsilon (t+ R) \|^p_{E_\epsilon} \leq \left( \frac{a_3}{R} + a_2 \right) \, e^{a_1}:= \beta_2 , \quad t \geq 0 .
\end{equation}

Choosing $R = T_1$ we have for $t \geq T_1$ that
$$
\| u^\epsilon (t) \|_{E_\epsilon} \leq \beta^{1/p}_2 :=\beta_1.
$$

\end{proof}

\begin{lemma}\label{lemma:Ebound-lim}
If $u(\cdot, u_0) \in C([0, \infty), L^2(\R^n))$ is the global solution of \eqref{eq:ev-abst-lim}, then there exist positive constants $T_1$ and $\beta_1$ such that
$$
\| u(t, u_0) \|_{E_0} \leq \beta_1, \quad \forall \; t \geq T_1.
$$
\end{lemma}
\begin{proof}
The proof is similar to the previous Lemma and will be omitted. 

\end{proof}

\begin{proposition}\label{cor:exis-attr-pert}
For each value of the parameter $\epsilon \in [0, 1)$ equation \eqref{eq:ev-abst} $(\eqref{eq:ev-abst-lim}$ if $\epsilon=0)$ generates a nonlinear $C^0$-semigroup $\{T_\epsilon(t)\}_{t\geq 0}$ in the space $L^2(\Omega^\epsilon)$ $(L^2(\R^n)$ if $\epsilon=0)$ defined by $T_\epsilon(t)u^\epsilon_0:= u^\epsilon(t;u^\epsilon_0)$ $(T_0(t)u_0:= u(t;u_0))$ which has global attractor $\mathcal{A}_\epsilon$. 
\end{proposition}
\begin{proof}
The first part of the Proposition is the statement of the Proposition \ref{teo:Brezis}. For the existence of attractors we just recall that by Lemma \ref{lemma:Ebound} (Lemma \ref{lemma:Ebound-lim} if $\epsilon=0$) the compact set $K_\epsilon = \overline{B^{E_\epsilon}(0,\beta_1)}^{L^2(\Omega^\epsilon)}$ ($K_0 = \overline{B^{E_0}(0,\beta_1)}^{L^2(\R^n)}$ if $\epsilon=0$) absorbs bounded sets of $L^2(\Omega^\epsilon)$ ($L^2(\R^n)$ if $\epsilon=0$) under $T_\epsilon(t)$ ($T_0(t)$ if $\epsilon=0$). By \cite[Chapter I, Theorem 1.1]{Teman:88} there exists a global attractor $\mathcal{A}_\epsilon \subset L^2(\Omega^\epsilon)$ ($\mathcal{A}_0 \subset L^2(\R^n)$ if $\epsilon =0$).

\end{proof}

\subsection{Upper-semicontinuity of pullback attractors}

We start remembering the definition of Hausdorff semi-distance between two subsets $A$ and $B$ of a metric space $(X,d)$:
\[
\operatorname{dist}_H(A,B) = \sup_{a\in A} \inf_{b\in B} d(a,b).
\]
We say that the family $\{{A}_\epsilon\}_{\epsilon} \subset X$ is upper-semicontinuous in $\epsilon=\epsilon_0$, if
$$
\lim_{\epsilon \to \epsilon_0} \operatorname{dist}_H({A}_\epsilon,  {A}_{\epsilon_0}) = 0.
$$

Now we prove that the family of global attractors $\{\mathcal{A}_\epsilon\}_{\epsilon \in [0,1]}$ is upper-semicontinuous in $\epsilon=0$. Since $\mathcal{A}_\epsilon$ and $\mathcal{A}_0$ belongs to different metrics spaces, such comparison is understood in the following sense
\begin{equation}\label{eq:upper-sem}
\lim_{\epsilon \to 0} \operatorname{dist}_{H_\epsilon}(\mathcal{A}_\epsilon, E_\epsilon \mathcal{A}_0)=0,
\end{equation}
where $\displaystyle \operatorname{dist}_{H_\epsilon}(\mathcal{A}_\epsilon, E_\epsilon \mathcal{A}_0):= \sup_{u_\epsilon \in \mathcal{A}_\epsilon} \inf_{v \in \mathcal{A}_0} |||u_\epsilon - E_\epsilon v|||_{L^2(\Omega^\epsilon)}$ and $E_\epsilon: L^2(\R^n) \to L^2(\Omega^\epsilon)$ was defined in \eqref{eq:oper-ext}. In the limiting case, for $A,B \subset L^2(\R^n)$ we will write  $\displaystyle \operatorname{dist}_{H_0}(A,B ):= \sup_{u \in A} \inf_{v \in B} |||u - v|||_{L^2(\R^n)}$.

First, if $u^\epsilon =u^\epsilon(\cdot, u^\epsilon_0)$ and $u =u(\cdot, u_0)$ are solutions of \eqref{eq:ev-abst} and \eqref{eq:ev-abst-lim} respectivelly, let be $w^\epsilon =u^\epsilon - E_\epsilon u$.  Thus $w^\epsilon_t  +  A_\epsilon u^\epsilon - E_\epsilon A_0 u = B_\epsilon(u^\epsilon) - E_\epsilon B(u)$. 

Since $a \geq 1$, it follows from Tartar's inequality the existence of $\alpha >0$ such that
\begin{align*}
\langle A_\epsilon u^\epsilon - E_\epsilon A_0 u^0, w^\epsilon \rangle_\epsilon & =  \langle  |\nabla u^\epsilon|^{p-2} \nabla u^\epsilon - | E_\epsilon \nabla u^0|^{p-2}  E_\epsilon \nabla u^0, \nabla w^\epsilon \rangle_\epsilon + \langle a |u^\epsilon |^{p-2}u^\epsilon - E_\epsilon \bar{a}| E_\epsilon u^0|^{p-2}  E_\epsilon u^0, w^\epsilon  \rangle_\epsilon \\
& =  \langle  |\nabla u^\epsilon|^{p-2} \nabla u^\epsilon - |\nabla E_\epsilon u^0|^{p-2} \nabla E_\epsilon u^0, \nabla w^\epsilon \rangle_\epsilon + \langle a (|u^\epsilon |^{p-2}u^\epsilon - |E_\epsilon u^0 |^{p-2}E_\epsilon u^0)  \\
 & \quad + (a - E_\epsilon \bar{a}) |E_\epsilon u^0|^{p-2} E_\epsilon u^0, w^\epsilon  \rangle_\epsilon \\
& \geq \alpha ( ||| \nabla w^\epsilon |||_{L^2(\Omega^\epsilon)}^p + ||| w^\epsilon |||_{L^2(\Omega^\epsilon)}^p ) +  \langle (a - E_\epsilon \bar{a}) |E_\epsilon u^0|^{p-2} E_\epsilon u^0, w^\epsilon  \rangle_\epsilon  .
\end{align*}

Therefore by Hölder's inequality 

\begin{align*}
\frac12 \frac{d}{dt} ||| w^\epsilon |||^2_{L^2(\Omega^\epsilon)}  & \leq  -  \epsilon^{-1} \int_{\Omega^\epsilon} (a_\epsilon - E_\epsilon \bar{a}) |E_\epsilon u^0|^{p-2} E_\epsilon u^0 w^\epsilon \, dxdy + ||| B(u^\epsilon) - E_\epsilon B(u^0) |||_{L^2(\Omega^\epsilon)} ||| w^\epsilon |||_{L^2(\Omega^\epsilon)} \\
& \leq \| a - E_\epsilon \bar{a} \|_{L^\infty (\Omega^\epsilon) } \epsilon^{-1} \int_{\Omega^\epsilon} \left( | E_\epsilon u^0 |^p + | E_\epsilon u^0 |^{p-1}  |u^\epsilon |  \right) dx  + L ||| w^\epsilon |||_{L^2 (\Omega^\epsilon)}^2 \\
& \leq \| a - E_\epsilon \bar{a} \|_{L^\infty (\Omega^\epsilon) }  \left( ||| E_\epsilon  u^0 |||_{L^p (\Omega^\epsilon)}^p + ||| E_\epsilon u^0 |||_{L^p (\Omega^\epsilon)}^{p-1}  ||| u^\epsilon |||_{L^p (\Omega^\epsilon)}  \right) + L ||| w^\epsilon |||_{L^2 (\Omega^\epsilon)}^2
\end{align*}

The uniform estimates given by Lemma \ref{lemma:Ebound} lead to
$$
\frac12 \frac{d}{dt} ||| w^\epsilon |||^2_{L^2(\Omega^\epsilon)} \leq M \| a - E_\epsilon \bar{a} \|_{L^\infty (\Omega^\epsilon) } + L ||| w^\epsilon |||_{L^2 (\Omega^\epsilon)}^2,
$$
in compact subsets of $\mathbb{R}$. Integrating this last inequality from $0$ to $t$, we obtain
$$
||| w^\epsilon |||^2_{L^2(\Omega^\epsilon)} \leqslant ||| u^\epsilon_0 - E_\epsilon u_0 |||_{L^2 (\Omega^\epsilon)}^2 + 2Mt \| a - E_\epsilon \bar{a} \|_{L^\infty (\Omega^\epsilon) } + 2L \int_0^t ||| w^\epsilon (s) |||^2_{L^2(\Omega^\epsilon)} \, ds .
$$

Hence, by Gronwall's Inequality

\begin{equation}\label{eq:continuity} 
||| w^\epsilon |||^2_{L^2(\Omega^\epsilon)} \leqslant \tilde{M}\left( ||| u^\epsilon_0 - E_\epsilon u_0 |||_{L^2 (\Omega^\epsilon)}^2 + \| a - E_\epsilon \bar{a} \|_{L^\infty (\Omega^\epsilon) } \right),
\end{equation}
in compact subsets of $\mathbb{R}$. 

From this discussion we derive the following Lemma

\begin{lemma}
Let $\{T_\epsilon(t)\}_{t\geqslant 0}$ and $\{T(t)\}_{t\geqslant 0}$ be the semigroups generated by the problems \eqref{eq:ev-abst} and \eqref{eq:ev-abst-lim} respectivelly. If $|||u^\epsilon_0 -E_\epsilon u_0|||_{L^2(\Omega^\epsilon)} \to 0$ then $||| T_\epsilon(t) u^\epsilon_0- E_\epsilon T(t)u^0 |||_{L^2(\Omega^\epsilon)} \to 0 $ uniformly for $t$ in compact subsets of $\R$.
\end{lemma}

\begin{corollary}
The family of global attractors $\{ \mathcal{A}_\epsilon \}_{\epsilon \in [0,1]}$ is upper-semicontinuous in $\epsilon = 0$.
\end{corollary}
\begin{proof}
Given $\delta >0$ let $T>0$ be such that $\operatorname{dist}_{H_0}(T(t)B,\mathcal{A}_0) < \frac{\delta}{2}$, for all $t \geq T$, where $\displaystyle B \supset \bigcup_{\epsilon \in [0,1)}  M_\epsilon \mathcal{A}_\epsilon$ is a bounded set (whose existence is guaranteed by Lemma \ref {lemma:Ebound}).

Now from \eqref{eq:continuity}, there exists $\epsilon_0 >0$ such that
\[
 \sup_{\xi_\epsilon \in \mathcal{A}_\epsilon}
|||T_\epsilon(t)\xi_\epsilon - E_\epsilon T_0(t)M_\epsilon\xi_\epsilon |||_{L^2(\Omega^\epsilon)}
<  \frac{\delta}{2},
\]
for all $\epsilon < \epsilon_0$. From the invariance of $\mathcal{A}_\epsilon$ under $T_\epsilon(t)$, $\epsilon \in [0,1)$, we obtain
\begin{align*}
\operatorname{dist}_{H_\epsilon}(\mathcal{A}_\epsilon,E_\epsilon \mathcal{A}_0)  & \leqslant \operatorname{dist}_{H_\epsilon}(T_\epsilon(t) \mathcal{A}_\epsilon,E_\epsilon T_0(t)M_\epsilon\mathcal{A}_\epsilon)  + \operatorname{dist}_{H_\epsilon}(E_\epsilon T_0(t) M_\epsilon \mathcal{A}_\epsilon,
E_\epsilon T_0(t)\mathcal{A}_0) \\
&  = \sup_{\xi_\epsilon \in \mathcal{A}_{\epsilon}}
 \operatorname{dist}_{H_\epsilon}(T_\epsilon(t)\xi_\epsilon,
 E_\epsilon T_0(t)M_\epsilon \xi_\epsilon)
 + \operatorname{dist}_{H_0}(T_0(t) M_\epsilon \mathcal{A}_\epsilon,
 \mathcal{A}_0) <  \frac{\delta}{2} +  \frac{\delta}{2},
  \end{align*}
which proves the upper-semicontinuity of the family of attractors.

\end{proof}


\bigskip \par\noindent {\bf Acknowledgements.} This work was partially supported by FAPESP $\#2012/06753-8$, FUNDUNESP $\#0135812$ and FUNDUNESP-PROPe $\#0019/008/13$, Brazil.


\end{document}